\documentclass[12pt,english]{amsart}
\usepackage[T1]{fontenc}
\usepackage[utf8]{inputenc}
\usepackage{geometry}
\usepackage{fancyhdr}
\usepackage{amstext}
\usepackage{amsthm}
\usepackage{amssymb}
\usepackage{esint}
\PassOptionsToPackage{normalem}{ulem}
\usepackage{ulem}

\makeatletter
\numberwithin{equation}{section}
\numberwithin{figure}{section}
\theoremstyle{plain}
\newtheorem{thm}{\protect\theoremname}
\theoremstyle{plain}
\newtheorem{prop}[thm]{\protect\propositionname}
\theoremstyle{plain}

\theoremstyle{remark}
\newtheorem*{rem*}{\protect\remarkname}
\theoremstyle{plain}
\newtheorem{lem}[thm]{\protect\lemmaname}

\usepackage[english]{babel}
\usepackage{fancyhdr}
\usepackage{tikz}
\usetikzlibrary{arrows,shapes}
\usetikzlibrary{calc,decorations.markings}
\tikzset{
  reversed with radius/.style={
    x radius=#1,
    y radius=-#1,
 }
}

\tikzset{
  with arrows/.style={
    decoration={ markings,
      mark=between positions #1 and .999 step #1 with {\arrow{stealth}}
    }, postaction={decorate}
  }, with arrows/.default=25mm,
} 

\providecommand{\lemmaname}{Lemma}
\providecommand{\propositionname}{Proposition}
\providecommand{\theoremname}{Theorem}
\providecommand{\corollaryname}{Corollary}
\newcommand{\abs}[1]{\ensuremath{|#1|}}

\newcommand{\Abs}[1]{\ensuremath{\left|#1\right|}}

\newcommand{\norm}[2]{\ensuremath{|\!|#1|\!|_{#2}}}

\newcommand{\Norm}[2]{\ensuremath{\left|\!\left|#1\right|\!\right|_{#2}}}

\newcommand{\cM}{\mathcal{M}}

\newcommand{\ccap}{{\rm cap}}

\usepackage{babel}
\providecommand{\lemmaname}{Lemma}
  \providecommand{\propositionname}{Proposition}
  \providecommand{\remarkname}{Remark}
\providecommand{\theoremname}{Theorem}

\usepackage[section]{placeins}

\usepackage{babel}
\providecommand{\lemmaname}{Lemma}
  \providecommand{\propositionname}{Proposition}
\providecommand{\theoremname}{Theorem}

\usepackage{babel}
\providecommand{\corollaryname}{Corollary}
\providecommand{\lemmaname}{Lemma}
\providecommand{\propositionname}{Proposition}
\providecommand{\remarkname}{Remark}
\providecommand{\theoremname}{Theorem}

\usepackage{babel}
\providecommand{\corollaryname}{Corollary}
\providecommand{\lemmaname}{Lemma}
\providecommand{\propositionname}{Proposition}
\providecommand{\remarkname}{Remark}
\providecommand{\theoremname}{Theorem}

\usepackage{babel}
\providecommand{\lemmaname}{Lemma}
\providecommand{\propositionname}{Proposition}
\providecommand{\remarkname}{Remark}
\providecommand{\theoremname}{Theorem}

\usepackage{babel}
\providecommand{\lemmaname}{Lemma}
\providecommand{\propositionname}{Proposition}
\providecommand{\remarkname}{Remark}
\providecommand{\theoremname}{Theorem}

\usepackage{babel}
\providecommand{\lemmaname}{Lemma}
\providecommand{\propositionname}{Proposition}
\providecommand{\remarkname}{Remark}
\providecommand{\theoremname}{Theorem}

\makeatother

\usepackage{babel}
\providecommand{\corollaryname}{Corollary}
\providecommand{\lemmaname}{Lemma}
\providecommand{\propositionname}{Proposition}
\providecommand{\remarkname}{Remark}
\providecommand{\theoremname}{Theorem}

\begin{document}
\sloppy

\title{Analytic capacities in Besov spaces}

\author{Anton Baranov}
\address{Department of Mathematics and Mechanics, St. Petersburg State University, 28, Universitetskii prosp., 
198504 Staryi Petergof, Russia}
\email{anton.d.baranov@gmail.com}
\author{Michael Hartz}
\address{Fachrichtung Mathematik, Universit\"at des Saarlandes, 66123 Saarbr\"ucken, Germany}
\email{hartz@math.uni-sb.de}
\thanks{M.H. was partially supported by the
Emmy Noether Program of the German Research Foundation (DFG Grant 466012782).
The work of I.K. in Sections 5 and 6 was supported by Russian Science Foundation (grant 23-11-00153). The work of R.Z. was supported by the pilot center Ampiric, funded by the France 2030 Investment Program operated by the Caisse des D\'ep\^ots.}
\author{Ilgiz Kayumov}
\address{Department of Mathematics and Computer Science, St. Petersburg State University, 
Russia, 14 line of the VO, 29B, 199178 St. Petersburg, Russia}
\email{ikayumov@gmail.com}
\author{Rachid Zarouf}
\address{Aix-Marseille University, Laboratoire ADEF, Campus Universitaire de Saint-Jérôme,
52 Avenue Escadrille Normandie Niemen, 13013 Marseille, France}
\address{Aix Marseille University, University of Toulon, CNRS, CPT, Marseille, France}
\email{rachid.zarouf@univ-amu.fr}

\subjclass[2010]{Primary 30H25; Secondary 30J10, 47A60}

\begin{abstract}
We derive new estimates on analytic capacities of finite sequences
in the unit disc in Besov spaces with zero smoothness, which 
sharpen the estimates obtained by N.\,K.~Nikolski in 2005 
and, for a range of parameters, are optimal.
The work is motivated both from the perspective of complex analysis by the description of sets of 
zeros/uniqueness, and from the one of matrix analysis/operator theory 
by estimates on norms of inverses.
\end{abstract}

\maketitle

\section{Introduction}

Let $\mathbb{D}=\left\{ z\in\mathbb{C}:\:\abs{z}<1\right\} $ be the
open unit disk, let $\mathbb{T}=\left\{ z\in\mathbb{C}:\:\abs{z}=1\right\} $
be its boundary and $\mathbb{D}_{*}= \mathbb{D}\setminus\{0\}$. We denote
by $\mathcal{H}ol(\mathbb{D})$ the space of analytic functions on
$\mathbb{D}$, equipped with the topology of local uniform convergence.
Let $X$ be a Banach space that is continuously contained in $\mathcal{H}ol(\mathbb{D})$ and that contains the polynomials.
Given a finite sequence $\sigma=(\lambda_{1},\dots,\lambda_{N})\in\mathbb{D}_*^N$,
Nikolski \cite{NN1} defined
the $X$-zero capacity of $\sigma$ as
\[
{\rm cap}_{X}(\sigma)=\inf\{\Norm{f}{X}\::\:f(0)=1,\:f\vert\sigma=0\},
\]
where $f\vert\sigma=0$ means that $f(\lambda_{i})=0$ for all $i=1,\dots,N$
taking into account possible multiplicities. Namely, if 
$\sigma=(\lambda_{1},\dots,\lambda_{1},\lambda_{2},\dots,\lambda_{2},\dots,\lambda_{s},\dots,\lambda_{s})
\in\mathbb{D}^{N}$,
where each $\lambda_{i}$ is repeated according to its multiplicity
$m_{i}\geq1,$ then $f\vert\sigma=0$ means that 
\[
f(\lambda_{i})=f'(\lambda_{i})=f''(\lambda_{i})=\dots=f^{(m_{i}-1)}(\lambda_{i})=0,
\qquad i=1,\dots, s.
\]
The latter quantity is closely related on one
hand to the problem of uniqueness sets for the function space $X$
and on the other hand to condition numbers of large matrices and inverses,
as observed by Nikolski
\cite[Section 1]{NN1}. We briefly review these connections here.

\subsection{\label{subsec:Motivation1}Motivation from complex analysis: sets of
zeros/uniqueness}

From the point of view of complex analysis, the $X$-zero capacities are closely related
to the problem of characterizing uniqueness sets for the function space $X$;
here $\sigma$ is said to be a uniqueness set for $X$ if $f\in X,\:f\vert\sigma=0\implies f=0$.
Following \cite{NN1}, assume that the function space $X$ satisfies the following
Fatou property: if $f_{n}\in X$, $\sup_{n}\norm{f_{n}}{X}<\infty$
and $\lim_{n\to\infty}f_{n}(z)=f(z)$ for $z\in\mathbb{D}$, then $f\in X$.
Then it is not hard to see that an infinite sequence $\sigma=(\lambda_{i})_{i\geq1}\in\mathbb{D}_{*}^{\infty}$
is a uniqueness sequence for $X$ if and only if 
\begin{equation}
\sup_{N}\{{\rm cap}_{X}(\sigma_{N})\}=\infty,\label{eq:criteria_uniqu_zero_sets}
\end{equation}
where $\sigma_{N}=(\lambda_{i})_{i=1}^{N}$ is the truncation of $\sigma$ of order
$N$. For example, let $X$ be the algebra $H^\infty$ of
bounded holomorphic functions in $\mathbb{D}$ endowed with the norm
$\norm{f}{H^{\infty}}=\sup_{\zeta\in\mathbb{D}}\abs{f(\zeta)}$.
It is known \cite[Theorem 3.12]{NN1} that given $\sigma_{N}=(\lambda_{i})_{i=1}^{N}\in\mathbb{D}_*^N$,
\begin{equation}
{\rm cap}_{H^{\infty}}(\sigma_{N})=\frac{1}{\prod_{i=1}^{N}\abs{\lambda_{i}}}.\label{eq:cap_Hardy}
\end{equation}
Denoting by 
\[
B=B_{\sigma_{N}}=\prod_{i=1}^{N}\frac{z-\lambda_{i}}{1-\overline{\lambda_{i}}z}
\]
the finite Blaschke product associated with $\sigma_{N},$ observe
that the right-hand side in \eqref{eq:cap_Hardy} is achieved by the
test function $f=B/B(0)$, which is admissible for the conditions
in the infimum defining the capacity of $\sigma_{N}$. Thus, an application
of the above criterion \eqref{eq:criteria_uniqu_zero_sets} leads to
the well-known Blaschke condition: an infinite sequence $\sigma=(\lambda_{i})_{i\geq1}\in\mathbb{D}_{*}^{\infty}$
is a uniqueness sequence for $H^{\infty}$ if and only if 
\[
\sum_{i\geq1}(1-\abs{\lambda_{i}})=\infty.
\]

\subsection{\label{subsec:Motivation2}Motivation in operator theory/matrix analysis}

 Let $T$ be an invertible operator
acting on a Banach space or an $N\times N$ invertible matrix with
complex entries acting on $\mathbb{C}^{N}$ equipped with some norm.
We seek upper bounds on the norm of the inverse $T^{-1}$.
Assume that the minimal polynomial of $T$ is given by 
\[
m(z)=m_{T}(z)=\prod_{i=1}^{N}(z-\lambda_{i}),
\]
where $\sigma=(\lambda_{i})_{i=1}^{N}\in\mathbb{D}_{*}^{N}$ and we
assumed for simplicity that $\deg m_{T}=N$.
Following \cite{NN1},
assume that our Banach space $X \subset \mathcal{H}ol(\mathbb{D})$ is an in fact
an algebra, and write $A = X$. Assume further that
\begin{enumerate}
\item $T$ admits a $C$-functional calculus on $A$, i.e.\
  there exists a bounded homomorphism $f \mapsto f(T)$ extending the polynomial functional calculus
  and a constant $C > 0$ such that
\[
\Norm{f(T)}{}\leq C\Norm{f}{A},\qquad f\in A;
\]
\item the shift operator $S:f\mapsto zf$, the backward shift operator
  $S^*: f \mapsto \frac{f - f(0)}{z}$ and the generalized backward shift operators
$f\mapsto\frac{f-f(\lambda)}{z-\lambda}$ are bounded on $A$ for all $\lambda \in \mathbb{D}$.
\end{enumerate}
These assumptions are mild and satisfied by all the algebras $A$
considered below. Noticing that the analytic polynomial $P=\frac{m(0)-m}{zm(0)}$
interpolates the function $\frac{1}{z}$ on $\sigma$ we observe that
\[
T^{-1}=P(T)=(P+mh)(T)
\]
 for any $h\in A$. Applying assumption (1) to the above operator
we obtain
\[
\norm{T^{-1}}{}\leq C\norm{P+mh}{A}
\]
and taking the infimum over all $h\in A$ and using our assumptions on $A$, we get
\begin{equation}
\norm{T^{-1}}{}\leq C\inf \Big\{\Norm{g}{A}: \ g\vert\sigma=P\vert\sigma=\frac{1}{z}\Big\vert\sigma \Big\}.\label{eq:sharp_est_cap_inverse}
\end{equation}
Now, if $f \in A$ satisfies $f(0) = 1$ and $f \vert \sigma = 0$, then
$g:=S^*(1-f) = - S^*(f)$ is admissible for the last infimum, and so
\begin{equation}
\norm{T^{-1}}{}\leq C\norm{S^{*}}{A\to A} {\rm cap}_{A}(\sigma).\label{eq:link_inverse_cap}
\end{equation}
In particular \eqref{eq:sharp_est_cap_inverse} and \eqref{eq:link_inverse_cap}
are applied (among other situations) in \cite{NN1} to the cases of: 
\begin{itemize}
\item Hilbert space contractions, $A$ the disc algebra and $C=1$;
\item Banach space contractions, $A$ the Wiener algebra of absolutely
convergent Taylor/Fourier series,
\[
A=W=\{f=\sum_{k\geq0}\hat{f}(k)z^{k}\in\mathcal{H}ol(\mathbb{D}):\:\Norm{f}{W}=\sum_{k\geq0}\abs{\hat{f}(k)}<\infty\},
\]
and once again $C=1$;
\item Tadmor--Ritt type matrices or power-bounded matrices on Hilbert spaces and
$A$ the Besov algebra 
\[
A=\mathcal{B}_{\infty,1}^{0}=\left\{ f\in\mathcal{H}ol(\mathbb{D}):\:\Norm{f}{\mathcal{B}_{\infty,1}^{0}}=\abs{f(0)}+\int_{0}^{1}\Norm{f_{\rho}'}{L^{\infty}(\mathbb{T})}{\rm d}\rho<\infty\right\} ,
\]
where $f_{\rho}(\zeta)=f(\rho\zeta),\:\zeta\in\mathbb{T}$. 
\end{itemize}

\subsection*{Outline of the paper}

In Section \ref{sec:Known-results-and} below we first review Nikolski's
upper estimates on ${\rm cap}_{X}(\sigma)$ where $X$ is a general
Besov space $\mathcal{B}_{p,q}^{s},$ $s\geq0,\:(p,q)\in[1,\infty]^{2}$,
see below for their definition. We also relate the special case $(p,q)=(\infty,1)$ to applications
in operator theory/matrix analysis and especially to Schäffer's question
on norms of inverses.

In Section \ref{sec:Main-results} we formulate the main results of the
paper. Theorem \ref{thm:low_bd_infty_1},
which corresponds to the special case $(p,q)=(\infty,1)$, exhibits
an explicit sequence $\sigma^{\star}$ for which we derive a quantitative
lower bound on ${\rm cap}_{\mathcal{B}_{\infty,1}^{0}}(\sigma^{\star})$
and thereby $almost$ prove the sharpness of Nikolski's upper bound
in this case.
Theorem \ref{thm:Upper_bds}  improves Nikolski's upper bounds
on ${\rm cap}_{\mathcal{B}_{p,q}^{0}}(\sigma)$ for 
a range of parameters, while in Theorem \ref{thm:low_bd_p_q} the sharpness of these new bounds is discussed. 

In Section \ref{sec:Proof-of-Lemma} we prove Lemma \ref{lem:lower_bound},
which is our main tool for bounding the capacities from below. 
In Section \ref{sec:Proof-of-Theorem1} we prove Theorem
\ref{thm:low_bd_infty_1}.
The proofs of the lower bounds in Theorem \ref{thm:low_bd_p_q} are
provided in Section \ref{sec:Proof-of-th_low_bds}.
Finally, in Section
\ref{sec:Proofs-of-th_up_bds} we prove  the upper
bounds stated in Theorem \ref{thm:Upper_bds}. The proof is based on estimates of Besov norms
of finite Blaschke products (Proposition \ref{bln})
which may be of independent interest.
\subsection*{Acknowledgements}
We would like to thank the referee for carefully reading our manuscript and for constructive and helpful comments.

The first author is a winner of the “Junior Leader” contest conducted by the Foundation for the Advancement of Theoretical Physics and Mathematics “BASIS” and would like to thank its sponsors and jury.

\bigskip


\section{\label{sec:Known-results-and}Known results and open questions}


\subsection{Capacities in Besov spaces}
The case where $X$ is an analytic Besov space $X=\mathcal{B}_{p,q}^{s}$
is considered in \cite{NN1}. Let $s\geq0,$ $1\leq p,\,q\leq\infty$
and let 
\[
  \mathcal{B}_{p,q}^{s}=\left\{ f\in\mathcal{H}ol(\mathbb{D}):\:\Norm{f}{\mathcal{B}_{p,q}^{s}}^{*}=\left(\int_{0}^{1}((1-\rho)^{m-s-1/q}\Norm{f_{\rho}^{(m)}}{L^{p}(\mathbb{T})})^q{\rm d}\rho\right)^{1/q}<\infty\right\} ,
\]
where $f_{\rho}^{(m)}(\zeta)=f^{(m)}(\rho\zeta),$ $m$ being a nonnegative
integer such that $m>s$ (the choice of $m$ is not essential and the norms for different $m$-s are equivalent).
We need to  make the obvious modification for $q = \infty$. The space $\mathcal{B}_{p,q}^{s}$ equipped
with the norm 
\[
\Norm{f}{\mathcal{B}_{p,q}^{s}}=\sum_{k=0}^{m-1}\abs{f^{(k)}(0)}+\Norm{f}{\mathcal{B}_{p,q}^{s}}^{*}
\]
is a Banach space. We refer to \cite{BL,Pe,Tr}
for general properties of Besov spaces. 
Note that for $1\leq q < \infty$ we have $f_\rho \to f$ in the norm of 
$\mathcal{B}_{p,q}^{s}$ as $\rho\to 1-$.

In the present paper we deal with Besov spaces with zero smoothness $s=0$. In this case we take $m=1$ and
$$
\Norm{f}{\mathcal{B}_{p,q}^{0}}^{*}=\left(\int_{0}^{1} (1-\rho)^{q-1} \Norm{f'_{\rho}}{L^{p}
(\mathbb{T})}^q {\rm d}\rho\right)^{1/q}, \qquad 1\le q<\infty,
$$
$$
\Norm{f}{\mathcal{B}_{p,\infty}^{0}}^{*} = \sup_{0<\rho <1} (1-\rho)\|f'_\rho\|_{L^{p}(\mathbb{T})}.
$$
Note that $\mathcal{B}_{\infty,\infty}^{0}$ coincides with the classical Bloch space.


It is shown \cite[Theorem 3.26]{NN1} that given $1\leq p,\,q\leq\infty$,
$s>0$ and $\sigma\in\mathbb{D}_{*}^{N}$ the following upper estimate
holds
\[
{\rm cap}_{\mathcal{B}_{p,q}^{s}}(\sigma)\leq c\frac{N^{s}}{\prod_{i=1}^{N}\abs{\lambda_{i}}},
\]
 where $c=c(s,q),$ and that if $s=0$ then 
\begin{equation}
{\rm cap}_{\mathcal{B}_{p,q}^{0}}(\sigma)\leq c\frac{\left(\log N\right)^{1/q}}{\prod_{i=1}^{N}\abs{\lambda_{i}}},\label{eq:Nik_log_up_bd}
\end{equation}
where $c>0$ is a numerical constant. It is also shown that for $s>0$
these estimates are asymptotically sharp in the following sense \cite[Theorem 3.31]{NN1}:
there exist constants $c=c(s,p,q)>0$ and $K=K(s,p,q)>0$ such that
for any $\sigma=(\lambda_{1},\dots,\lambda_{N})\in\mathbb{D}_{*}^{N}$, $s>0$, $1\leq p,q\leq\infty$,
\[
{\rm cap}_{\mathcal{B}_{p,q}^{s}}(\sigma)\geq c\frac{N^{s}}{\prod_{i=1}^{N}\abs{\lambda_{i}}}\left(1+K-\prod_{i=1}^{N}(1+\abs{\lambda_{i}})\right).
\]
The sharpness of the upper bound in \eqref{eq:Nik_log_up_bd} is
left as an open question in \cite{NN1}. 

\subsection{Norms of inverses and Sch\"affer's question}

Let $\Norm{\cdot}{}$ denote the operator norm
induced on $\cM_{N}$, the space of complex $N \times N$ matrices, by a Banach space norm on $\mathbb{C}^N$. What is the
smallest constant $\mathcal{S}_{N}$ so that 
\[
\abs{\det{T}} \cdot \Norm{T^{-1}}{}\leq\mathcal{S}_{N}\Norm{T}{}{}^{N-1}
\]
holds for any invertible matrix $T\in\cM_{N}$ and any operator
norm $\Norm{\cdot}{}$?
Schäffer \cite[Theorem 3.8]{SJ} proved that 
\[
\mathcal{S}_{N}\leq\sqrt{eN},
\]
but he conjectured that $\mathcal{S}_N$ should in fact be bounded, as it is the case for Hilbert space.
This conjecture was disproved in the early 90's by E.~Gluskin,
M.~Meyer, and A.~Pajor \cite{GMP}.
Later, Queffélec~\cite{Qu4} showed that the $\sqrt{N}$ bound is essentially optimal
for arbitrary Banach spaces, but both arguments are non-constructive.
An explicit construction giving a $\sqrt{N}$ lower bound was recently given
in  \cite{SZ1}. For a detailed account on the history of Sch\"affer's question,
the reader is referred to \cite{SZ1}.
A key tool in the works cited above is the equality
\begin{align}
\mathcal{S}_{N}=\sup_{\left(\lambda_{1},\dots,\lambda_{N}\right)\in\mathbb{D}^{N}}
\prod_{i=1}^{N}\abs{\lambda_{i}}\left({\rm cap}_{W}\left(\lambda_{1},\dots,\lambda_{N}\right)-1\right),
\label{GMP_lemma}
\end{align}
due to Gluskin, Meyer and Pajor. It connects Sch\"affer's question to capacity in the Wiener algebra and shows
that \eqref{eq:link_inverse_cap} is essentially sharp in this case.

It is natural to consider Sch\"affer's question for operator classes different from Hilbert
or Banach space contractions. In particular, following \cite{NN1},
we may consider the following classes, which admit a Besov $\mathcal{B}^0_{\infty,1}$-functional calculus.

\begin{enumerate}
\item Power bounded operators on Hilbert space, i.e.\ operators $T$ on Hilbert space
  satisfying
\[
\sup_{k\geq0}\Norm{T^{k}}{}=C_{pb} <\infty.
\]
Peller \cite{VP} proved that $\Norm{f(T)}{}\leq k_{G}C_{pb}^{2}\Norm{f}{\mathcal{B}_{\infty,1}^{0}}$
for every analytic polynomial $f$, where $k_{G}$ is the Grothendieck constant.
Combining \eqref{eq:link_inverse_cap} with
Nikolski's upper estimate \eqref{eq:Nik_log_up_bd} for $q=1$, we obtain the upper bounds
\begin{equation}
\Norm{T^{-1}}{}\leq c_{1}\cdot{\rm cap}_{\mathcal{B}_{\infty,1}^{0}}\left(\lambda_{1},\dots,\lambda_{N}\right)\leq c_{3}\frac{k_{G}C_{pb}^{2}\log N}{\prod_{i=1}^{N}\abs{\lambda_{i}}},\label{eq:PB_hilb_schaf}
\end{equation}
where $c_{1}>0$ is an absolute constant and $(\lambda_{i})_{i=1}^{N}$
is the sequence of eigenvalues of $T$.
\item Tadmor--Ritt operators on Banach space, i.e.\ operators $T$ acting on a Banach space
and satisfying the resolvent estimate
\[
\sup_{\abs{\zeta}>1}\abs{\zeta-1}\Norm{(\zeta-T)^{-1}}{}=C_{TR} <\infty.
\]
According to P. Vitse's functional calculus \cite[Theorem 2.5]{PV}
we have $\Norm{f(T)}{}\leq300C_{TR}^{5}\Norm{f}{\mathcal{B}_{\infty,1}^{0}}$
for every analytic polynomial $f$, and following the same reasoning
as above this yields 
\begin{equation}
\Norm{T^{-1}}{}\leq c_{2}\cdot{\rm cap}_{\mathcal{B}_{\infty,1}^{0}}\left(\lambda_{1},\dots,\lambda_{N}\right)\leq c_{2}\frac{300C_{TR}^{5}\log N}{\prod_{i=1}^{N}\abs{\lambda_{i}}},\label{eq:TR_Schaf}
\end{equation}
where $c_{2}>0$ is an absolute constant. 
In fact, thanks to work of Schwenninger \cite{Schwenninger}, the dependence on $C_{TR}$ can
be improved from $C_{TR}^5$ to $C_{TR} (\log C_{TR} + 1)$.
\end{enumerate}
The sharpness of the right-hand side in \eqref{eq:PB_hilb_schaf}
and \eqref{eq:TR_Schaf} is an open question both from the point of
view of operators/matrices and from the one of capacities.
Note that we have the following (strict) inclusions:
\begin{equation}
W\subset \mathcal{B}_{\infty,1}^{0}\subset H^{\infty}\label{eq:embeddings}
\end{equation}
(see \cite{BL,Pe} or \cite[Section B.8.7]{NN2}). Observe that $\mathcal{B}_{\infty,1}^{0}$
is actually contained in the disc algebra. From the perspective of
capacities \eqref{eq:embeddings} implies that for any sequence $\sigma=\left(\lambda_{1},\dots,\lambda_{N}\right)\in\mathbb{D}_*^N$
we have
\begin{equation}
{\rm cap}_{H^{\infty}}(\sigma)\leq c_{3}{\rm cap}_{\mathcal{B}_{\infty,1}^{0}}(\sigma)\leq c_{4}{\rm cap}_{W}(\sigma)\label{eq:comparison_cap}
\end{equation}
where $c_3, c_4>0$ are absolute constants. Observe that in view of  \eqref{eq:comparison_cap} and \eqref{GMP_lemma} any sequence 
$\sigma\in\mathbb{D}_*^N$ such that $\prod_{i=1}^{N}\abs{\lambda_{i}} \cdot {\rm cap}_{\mathcal{B}_{\infty,1}^{0}}(\sigma)$
grows unboundedly in $N$ will automatically give a counterexample  to Schäffer's original question.

\section{\label{sec:Main-results}Main results }

Throughout this paper, we will use the following standard notation.
For two positive functions $f,g$
we say that $f$ is dominated by $g$, denoted by $f\lesssim g$,
if there is a constant $c>0$ such that $f\leq cg$ for all admissible variables. We say that $f$
and $g$ are comparable, denoted by $f\asymp g$, if both $f\lesssim g$ and $g\lesssim f$. 

The main goals of this paper are to 
\begin{enumerate}
\item
Provide an example of a sequence $\sigma^{\star}=\left(\lambda_{1},\dots,\lambda_{N}\right)\in\mathbb{D}_{*}^{N}$
such that  $\prod_{i=1}^{N}\abs{\lambda_{i}} \cdot {\rm cap}_{\mathcal{B}_{\infty,1}^{0}}(\sigma^{\star})$
$almost$ (up to a double logarithmic factor) approaches Nikolski's upper bound $\log N$.  
\item 
Improve Nikolski's upper bound \eqref{eq:Nik_log_up_bd} on $\prod_{i=1}^{N}\abs{\lambda_{i}}\cdot 
{\rm cap}_{\mathcal{B}_{p,q}^{0}}(\sigma)$
identifying three regions of $(p,q)\in[1,\infty]^{2}$ with a different behavior of this quantity (see Theorem \ref{thm:Upper_bds} below).
For all $(p,q)$ with $p\ne \infty$ our estimates give a smaller growth than the estimates in \cite{NN1}, and for a range 
of parameters, namely for $1\leq q\leq p <\infty$ and $p\geq2$, they are best possible.
\end{enumerate}


\subsection{\label{subsec:The-special-case}
A lower estimate on ${\rm cap}_{\mathcal{B}_{\infty,1}^{0}}(\sigma)$}

Our approach to bounding ${\rm cap}_{\mathcal{B}_{\infty,1}^{0}}(\sigma)$
from below
uses duality.
To estimate
${\rm cap}_{\mathcal{B}_{\infty,1}^{0}}(\sigma)$ from below, we estimate the
Besov seminorm in $\mathcal{B}_{1,\infty}^{0}$
of finite Blaschke products from above. The key inequality, which will be proved in Lemma \ref{lem:lower_bound}, is
\begin{equation}
{\rm cap}_{\mathcal{B}_{\infty,1}^{0}}(\sigma)\gtrsim\frac{1}{\prod_{i=1}^{N}\abs{\lambda_{i}}}\frac{1-\prod_{i=1}^{N}\abs{\lambda_{i}}^2}{\Norm{B}{\mathcal{B}_{1,\infty}^{0}}^{*}},
\label{eq:Bourgain_analog}
\end{equation}
where $\sigma=\left(\lambda_{1},\dots,\lambda_{N}\right)$ is an arbitrary
sequence in $\mathbb{D}_{*}^{N}$, and $B=B_{\sigma}$ is the finite
Blaschke product associated to $\sigma$. To conclude we consider
$n\geq2$ and for $k=1,\dots,n$ we put
\[
\sigma_{k}=(r_{k}^{(n)}e^{2i\pi j/2^{k}})_{j=1}^{2^{k}}\in \mathbb{D}_*^{2^k},
\qquad
r_{k}^{(n)}=\left(1-1/n\right)^{2^{-k}}.
\]
We put $N=\sum_{k=1}^{n}2^{k}\asymp2^{n}$ and define the sequence
$\sigma^{\star}=\left(\lambda_{1},\dots,\lambda_{N}\right)\in\mathbb{D}_{*}^{N}$
by 
\begin{equation}
\sigma^{\star}=(\sigma_{1},\,\sigma_{2},\,\dots,\sigma_{n}).
\label{eq:sigma_star}
\end{equation}
Denoting by $B^{\star}$ the Blaschke product associated with $\sigma^{\star}$
we have
\begin{equation}
\label{bsta}
B^{\star}(z)=\prod_{k=1}^{n}\frac{z^{2^{k}}-a}{1-az^{2^{k}}},
\end{equation}
where $a=1-\frac{1}{n}.$ We will prove the following result.

\begin{prop}
\label{prop:1}
The Blaschke product $B^{\star}$ satisfies
\begin{equation}
\|B^{\star}\|_{\mathcal{B}_{1,\infty}^{0}}^{*}\lesssim\frac{\log\log N}{\log N}.\label{eq:1}
\end{equation}
\end{prop}

Taking into account that $\prod_{j=1}^{N}|\lambda_{j}|\le e^{-1}$
and combining \eqref{eq:Bourgain_analog} with \eqref{eq:1} we obtain
the following theorem.

\begin{thm}
\label{thm:low_bd_infty_1} Let $\sigma^{\star}\in\mathbb{D}_{*}^{N}$ and 
$B^{\star}$ be defined by \eqref{eq:sigma_star} and \eqref{bsta}. Then 
\[
\prod_{i=1}^{N}|\lambda_{i}|\cdot{\rm cap}_{\mathcal{B}_{\infty,1}^{0}}(\sigma^{\star})\gtrsim\frac{\log N}{\log\log N}.
\]
\end{thm}

As a consequence regarding Schäffer's question, Theorem \ref{thm:low_bd_infty_1}
implies (taking into account \eqref{eq:comparison_cap}) that
\[
\prod_{i=1}^{N}|\lambda_{i}|\cdot{\rm cap}_{W}(\sigma^{*})\gtrsim\frac{\log N}{\log\log N}.
\]
From this, following arguments in \cite{SZ1}, one obtains another explicit counterexample to Sch\"affer's question,
acting as multiplication by $z$ on the quotient $W / B^{\star} W$
of the Wiener algebra. One can identify the dual space of $W/ B^{\star} W$
with the space of rational functions of degree at most $N$
with poles at $1/\bar{\lambda}_{j}$
for $j=1,\dots, N$, equipped with the supremum norm of the Taylor coefficients.
Then, as in \cite[Theorem 8]{SZ1}, one obtains another explicit matrix that serves as a counterexample
to Sch\"affer's question.


\subsection{\label{subsec:Upper-bounds-on} Upper bounds on ${\rm cap}_{\mathcal{B}_{p,q}^{0}}(\sigma)$
for general values of $(p,q)\in[1,\infty]^{2}$}

In the following statements the constants in $\lesssim$ relations may depend on $p, q$, but not on $N$.

\begin{thm}
\label{thm:Upper_bds}
Given $(p,q)\in[1,\infty]^{2}$ and $\sigma=\left(\lambda_{1},\dots,\lambda_{N}\right)\in\mathbb{D}_{*}^{N}$,
the following upper estimates on ${\rm cap}_{\mathcal{B}_{p,q}^{0}}(\sigma)$ hold
depending on the region to which $(p,q)$ belongs. 

1) If $(p,q)\in[1,2]^{2}$ \textup(Region I\textup), then 
\[
{\rm cap}_{\mathcal{B}_{p,q}^{0}}(\sigma)\lesssim\frac{(\log N)^{1/q-1/2}}{\prod_{i=1}^{N}\abs{\lambda_{i}}}.
\]

2) If $1\leq p\leq q\leq\infty$ and $q\geq2$ \textup(Region II\textup), then 
\[
{\rm cap}_{\mathcal{B}_{p,q}^{0}}(\sigma)\lesssim\frac{1}{\prod_{i=1}^{N}\abs{\lambda_{i}}}.
\]

3) If $1\leq q\leq p\leq\infty$ and $p\geq2$ \textup(Region III\textup), then 
\[
{\rm cap}_{\mathcal{B}_{p,q}^{0}}(\sigma)\lesssim\frac{(\log N)^{1/q-1/p}}{\prod_{i=1}^{N}\abs{\lambda_{i}}}.
\]
\end{thm}

\begin{rem*}
The upper bound in part 2 of Theorem \ref{thm:Upper_bds}  is attained by
any sequence $\sigma=\left(\lambda_{1},\dots,\lambda_{N}\right)\in\mathbb{D}_{*}^{N}$
such that $\abs{\lambda_{i}}\ge 1-1/N$ for all $i=1,\dots, N$.
\end{rem*}


\subsection{\label{subsec:Quantitive-lower-estimates} Lower estimates
on ${\rm cap}_{\mathcal{B}_{p,q}^{0}}(\sigma^{\star})$}

In the following theorem we derive quantitive lower estimates on ${\rm cap}_{\mathcal{B}_{p,q}^{0}}(\sigma^{\star})$
for $1\leq q\leq p\leq\infty$. This proves,
in particular, the sharpness of Theorem \ref{thm:Upper_bds} for $(p,q)$ in Region III if $p <\infty$.

\begin{thm}
\label{thm:low_bd_p_q} 
Let $\sigma^{\star}\in\mathbb{D}_{*}^{N}$ and $B^{\star}$ be defined by \eqref{eq:sigma_star} and \eqref{bsta},
and let $(p,q)\in[1,\infty]^{2}$ be such that $1\leq q\leq p\leq\infty$. Then
\begin{align}
  \prod_{i=1}^{N}|\lambda_{i}|\cdot{\rm cap}_{\mathcal{B}_{p,q}^{0}}(\sigma^{\star})&\gtrsim(\log N)^{1/q-1/p},
\qquad p<\infty,
\label{eq:low_bd_p_q} \\
  \prod_{i=1}^{N}|\lambda_{i}|\cdot{\rm cap}_{\mathcal{B}_{\infty,q}^{0}}(\sigma^{\star}) &\gtrsim
\frac{(\log N)^{1/q}}{\log\log N}.
\label{tte1}
\end{align}
In particular, for $(p,q)$ in Region III and $p<\infty$, 
\begin{equation}
\prod_{i=1}^{N}|\lambda_{i}|\cdot{\rm cap}_{\mathcal{B}_{p,q}^{0}}(\sigma^{\star})\asymp(\log N)^{1/q-1/p}.
\label{tte2}
\end{equation}
\end{thm}

However, for  $1\leq q\leq p < 2$ there is still a certain gap between the upper and lower estimates for the capacities:
\[
(\log N)^{1/q-1/p}\lesssim\prod_{i=1}^{N}|\lambda_{i}|\cdot{\rm cap}_{\mathcal{B}_{p,q}^{0}}(\sigma^{\star})\lesssim(\log N)^{1/q-1/2}.
\]
Let us consider the diagonal case $1 \le q = p < 2$. Rudin \cite{Rudin55}
showed that there exists a Blaschke product that is not contained in $\mathcal{B}_{1,1}^0$, see also \cite{Piranian68}.
Vinogradov \cite[Theorem 3.11]{Vinogradov95} extended Rudin's result to $\mathcal{B}_{p,p}^0$ for $p \in (0,2)$.
These results perhaps suggest that the expression in the middle might be unbounded for $1 \le q = p < 2$.
Indeed, unboundedness would follow if we knew that there are Blaschke sequences that are not zero sets for $\mathcal{B}_{p,p}^0$.
However, the existence of such Blaschke sequences appears to be an open question.
Results about zero sets for $\mathcal{B}_{p,p}^0$, also for $p > 2$, can be found in \cite{GP06}.

Instead, we will give a different, qualitative argument showing that, in case $1 \le q=p < 2$, the expression in the middle may be unbounded.
\begin{thm}
  \label{thm:Bpp_lower_bound}
  For each $N \in \mathbb{N}$ there exists a finite sequence $\sigma_N \in \mathbb{D}_*^N$
  such that for all $1 \le p < 2$, we have
  \begin{equation*}
    \lim_{N \to \infty} \prod_{\lambda \in \sigma_N} |\lambda| \cdot {\rm cap}_{\mathcal{B}^0_{p,p}}(\sigma_N) = \infty.
  \end{equation*}
\end{thm}
 
It will be convenient to extend the definition of ${\rm cap}_{\mathcal{B}^s_{p,q}}(\sigma)$
to possibly infinite sequences $\sigma$ in the obvious way. The infimum over the empty
set is understood to be $+ \infty$, so that ${\rm cap}_{\mathcal{B}^s_{p,q}}(\sigma) = + \infty$
in case $\sigma$ is a uniqueness set for $\mathcal{B}_{p,q}^s$.
Our approach to bound ${\rm cap}_{\mathcal{B}_{p,q}^{0}}(\sigma)$
from below is based on a duality method. Namely, the key step of the proof is the following lemma:

\begin{lem}
\label{lem:lower_bound} Given $1\leq p,q\leq\infty$
and a finite
sequence $\sigma$ in $\mathbb{D}_*$,
we have 
\[
  \prod_{\lambda \in \sigma}|\lambda|\cdot{\rm cap}_{\mathcal{B}_{p,q}^{0}}(\sigma)\gtrsim\frac{1-\prod_{\lambda \in \sigma}|\lambda|^2}{\Norm{B_{\sigma}}{\mathcal{B}_{p',q'}^{0}}^{*}},
\]
where $B_{\sigma}$ is the Blaschke product with the zero set $\sigma$
and $p',q'$ are the exponents conjugate to $p,q$.
The same estimate is true for arbitrary Blaschke sequences $\sigma$ in $\mathbb{D}_*$ in case $1 \le p=q \le 2$.
\end{lem}

To prove the lower estimate \eqref{eq:low_bd_p_q} it
remains to apply Lemma \ref{lem:lower_bound} to $\sigma=\sigma^{\star}$
and estimate from above the Besov seminorm of $B^{\star}$. Namely we prove the following.

\begin{prop}
\label{jer} 
If $1\le p\le q \le \infty$, then
\begin{align*}
  \|B^{\star}\|_{\mathcal{B}_{p,q}^{0}}^{*} &\lesssim\frac{1}{(\log N)^{1/p-1/q}}, \qquad p>1, \\
  \|B^{\star}\|_{\mathcal{B}_{1,q}^{0}}^{*} &\lesssim \frac{\log\log N}{(\log N)^{1-1/q}}.
\end{align*}
\end{prop}

The idea of the proof of Theorem \ref{thm:Bpp_lower_bound} is also to use duality.
In case $p=1$, the dual norm turns out to be the Bloch semi-norm.
An obstacle to this strategy is a result of Baranov, Kayumov, and Nasyrov \cite{BKN22},
according to which the Bloch semi-norm of finite Blaschke products is bounded below by a universal constant.
Instead, we will work with infinite Blaschke products, and carry out an approximation argument.
\bigskip


\section{\label{sec:Proof-of-Lemma}Proof of Lemma \ref{lem:lower_bound}}

We first prove Lemma \ref{lem:lower_bound}.
Let $\left\langle \cdot,\,\cdot\right\rangle $ denote the Cauchy sesquilinear form:
given two functions $g \in H^p$ and $h \in H^{p'}$, let
\begin{equation*}
  \langle h,g \rangle = \int_{\mathbb{T}} h(z) \overline{g(z)} \, dm (z),
\end{equation*}
where $m$ denotes the normalized Lebesgue measure on $\mathbb{T}$.
We require the following basic duality result for Besov spaces.

\begin{lem}
  \label{lem:duality}
  Let $1 \le p,q \le \infty$. There exists a constant $C \ge 0$ such that for all functions $f$ and $g$ that are analytic in a neighborhood of $\overline{\mathbb{D}}$, we have
  \begin{equation*}
    |\langle f,g \rangle|  \le |f(0)| |g(0)| + C \|f\|_{\mathcal{B}^0_{p,q}}^* \|g\|_{\mathcal{B}^0_{p',q'}}^*,
  \end{equation*}
  where $p',q'$ are the exponents conjugate to $p,q$.
\end{lem}

\begin{proof}
  
Denote by $(h,\,g)$ the scalar product on the Bergman space $A^{2}$
defined by 
\[
(h,\,g)=\int_{\mathbb{D}}h(u)\overline{g(u)}{\rm d}\mathcal{A}(u),\qquad h,\,g\in A^{2},
\]
where ${\rm d}\mathcal{A}(u) = \frac{{\rm d}x\,{\rm d}y}{\pi}$ is the normalized planar Lebesgue measure on $\mathbb{D}$.
We recall the simplest form of Green's formula, 
\begin{equation}
\left\langle \phi,\,\psi\right\rangle =(\phi',\,S^{*}\psi)+\phi(0)\overline{\psi(0)},
\label{gr}
\end{equation}
where $S^{*}$ is the backward shift operator $S^{*}f=(f-f(0))/z$ and $\varphi,\psi$ are functions
that are analytic in a neighborhood of $\overline{\mathbb{D}}$.
We will also need to use the following integral formula. Recall that
the fractional differentiation operator $D_{\alpha}$, $-1<\alpha<\infty$,
is defined by $D_{\alpha}(z^{j})=\frac{\Gamma(j+2+\alpha)}{(j+1)!\Gamma(2+\alpha)}z^{j}$,
$j=0,\,1,\,2,\,\dots$, and extends linearly and continuously to the whole space $\mathcal{H}ol(\mathbb{D})$. 
Then, for functions $f,g$ analytic
in a neighborhood of $\overline{\mathbb{D}}$ and $-1<\alpha<\infty$, we have 
\begin{equation}
\int_{\mathbb{D}}f(u)\overline{g(u)}{\rm d}\mathcal{A}(u)=(\alpha+1)\int_{\mathbb{D}}D_{\alpha}f(u)
\overline{g(u)}\left(1-\left|u\right|^{2}\right)^{\alpha}{\rm d}\mathcal{A}(u),
\label{gr2}
\end{equation}
see \cite[Lemma 1.20]{HKZ}. 

Let $f,g$ be analytic in a neighborhood of $\overline{\mathbb{D}}$.
Applying \eqref{gr} we get 
\[
  \left\langle f,\,g\right\rangle =(f',\,S^{*}g)+ f(0) \overline{g(0)}.
\]
Then we apply \eqref{gr2} to $\overline{(f',\,S^{*}g)}=(S^{*}g,\,f')$ with $\alpha=1$:
\begin{align*}
(S^{*}g,\,f') & =2\int_{\mathbb{D}}D_{1}(S^{*}g)(u)\overline{f'(u)}\left(1-\left|u\right|^{2}\right){\rm d}\mathcal{A}(u)\\
 & =2\int_{0}^{1}\rho\left(1-\rho^{2}\right)\left(\int_{\mathbb{T}}D_{1}(S^{*}g)(\rho z)\overline{f'(\rho z)}{\rm d}m(z)\right){\rm d}\rho.
\end{align*}
By Hölder's inequality  
\[
\Abs{\int_{\mathbb{T}}D_{1}(S^{*}g)(\rho z)\overline{f'(\rho z)}{\rm d}m(z)}\leq\|f_{\rho}'\|_{L^{p}}\|(D_{1}(S^{*}g))_{\rho}\|_{L^{p'}}.
\]
Since $D_1(S^* g) = \frac{1}{2}( S^* g + g')$ and $(S^* g)(z) = \frac{1}{z} \int_{0}^1 t g'(t z) \, dt$, it follows that
$\|D_1(S^* g)\|_{L^{p'}} \lesssim \|g'\|_{L^{p'}}$.
The preceding estimates therefore give
\[
\abs{(S^{*}g,\,f')}\lesssim\int_{0}^{1}\left(1-\rho\right)\|f_{\rho}'\|_{L^{p}}\|g_{\rho}'\|_{L^{p'}}{\rm d}\rho.
\]
Then (again by Hölder's inequality) we get
\[
\abs{(S^{*}g,\,f')}\lesssim\Norm{f}{\mathcal{B}_{p,q}^{0}}^{*}\Norm{g}{\mathcal{B}_{p',q'}^{0}}^{*},
\]
as desired.
\end{proof}

\begin{proof}[Proof of Lemma \ref{lem:lower_bound}]
  Suppose first that $\sigma$ is a finite sequence in $\mathbb{D}_*$, say $|\sigma| = N$.
  Let $f$ be a function that is analytic in a neighborhood of $\overline{\mathbb{D}}$ such that
  $f(0) = 1$ and $f \big|{\sigma} = 0$. Then we have (writing $B = B_\sigma$)
  \begin{equation*}
    \langle f,B \rangle  = \frac{f(0)}{B(0)} = \frac{1}{\prod_{\lambda \in \sigma} \lambda}.
  \end{equation*}
  On the other hand, Lemma \ref{lem:duality} shows that
  \begin{equation*}
    | \langle f, B \rangle | \le |f(0)|  |B(0)| + C \|f\|^*_{\mathcal{B}^0_{p,q}} \|B\|^*_{\mathcal{B}^0_{p',q'}}
    = \prod_{\lambda \in \sigma} |\lambda| + C \|f\|_{\mathcal{B}^0_{p,q}}^* \|B\|_{\mathcal{B}^0_{p',q'}}^*.
  \end{equation*}
  Thus,
  \begin{equation}
    \label{eqn:lower_bound_cap}
    \prod_{\lambda \in \sigma} |\lambda| \cdot \|f\|_{\mathcal{B}^0_{p,q}}^*
    \ge \frac{1 - \prod_{\lambda \in \sigma} |\lambda|^2}{C\|B\|_{\mathcal{B}^0_{p',q'}}^*}.
  \end{equation}
  Now, let $f \in \mathcal{B}^0_{p,q}$ be an arbitrary function such that $f(0) = 1$
  and $f \big|{\sigma} = 0$. Let $0 < r < 1$ be such that $\frac{1}{r} \sigma \subset \mathbb{D}$.
  Then $f_r$ vanishes on $\frac{1}{r} \sigma$, hence by what has already been proved,
  \begin{equation*}
    r^N \prod_{\lambda \in \sigma} |\lambda| \cdot \|f_r\|_{\mathcal{B}^0_{p,q}}^*
    \ge \frac{1 - r^{2N} \prod_{\lambda \in \sigma} |\lambda|^2}{C\|B_{\frac{1}{r} \sigma}\|_{\mathcal{B}^0_{p',q'}}^*}.
  \end{equation*}
  Recall that $\|f_r\|_{\mathcal{B}^0_{p,q}}^* \le \|f\|_{\mathcal{B}^0_{p,q}}$.
  Moreover, $B_{\frac{1}{r} \sigma}$ converges to $B_{\sigma}$ uniformly in a neighborhood of $\overline{\mathbb{D}}$ as $r \to 1$.
  So taking the limit $r \to 1$, we conclude that \eqref{eqn:lower_bound_cap} holds for arbitrary
  $f \in \mathcal{B}^0_{p,q}$ satisfying $f(0) = 1$ and $f \big|\sigma = 0$. Taking the infimum over all admissible functions $f$,
  we obtain the lemma for finite sequences.

  Let now $1 \le p=q \le 2$ and let $\sigma$ be a possibly infinite Blaschke sequence.
  Let $B = B_\sigma$ and let $f \in \mathcal{B}^0_{p,p}$ be a function vanishing on $\sigma$ with $f(0) = 1$.
  We apply Lemma \ref{lem:duality} to the functions $f_r$ and $B_r$ to obtain the bound
  \begin{equation*}
    |\langle f_r,B_r \rangle| \le |B(0)| + C \|f\|_{\mathcal{B}^0_{p,p}}^* \|B\|^*_{\mathcal{B}^0_{q,q}}
  \end{equation*}
  for all $r < 1$.

  The classical Littlewood--Paley inequality shows that $\mathcal{B}^0_{p,p} \subset H^p \subset H^1$
  (see \cite[Theorem 6]{LP36} and also \cite[Lemma 1.4]{Vinogradov95}), so $f_r \to f$ in the norm of $H^1$. Moreover, $B \in H^\infty$
  and  $B_r \to B$ weak-$*$ in $H^\infty$. From this, it follows that
  \begin{equation*}
    \langle f,B \rangle - \langle f_r,B_r \rangle = \langle f, B - B_r \rangle + \langle f - f_r, B_r \rangle  \to 0
  \end{equation*}
  as $r \to 1$. Thus,
  \begin{equation*}
    |\langle f,B  \rangle| \le |B(0)| + C \|f\|_{\mathcal{B}^0_{p,p}}^* \|B\|_{\mathcal{B}^0_{p',p'}}^*.
  \end{equation*}
  Using that $f \in H^1$ vanishes on $\sigma$, we may factor $f = B g$ for some $g \in H^1$. Then
  \begin{equation*}
    \langle f,B \rangle = \langle g,1 \rangle = g(0) = \frac{1}{B(0)}.
  \end{equation*}
  Combining the last two formulas and taking the infimum over all admissible $f \in \mathcal{B}^0_{p,p}$ again yields the desired inequality.
\end{proof}
\bigskip

\section{\label{sec:Proof-of-Theorem1}Proof of Theorem \ref{thm:low_bd_infty_1}}

\subsection{Proof of Proposition \ref{prop:1}}
For simplicity we write $B$ instead of $B^{\star}$ throughout
the proof. Then $N={\rm deg}\,B \asymp2^{n}$.
For the zeros $z_{1},\dots z_{N}$ of $B$ we have 
\[
\prod_{j=1}^{N}|z_{j}|=a^{n}<e^{-1}.
\]
For $z\in\mathbb{D}$, $|z|=r$, we have
\begin{equation}
|B'(z)|\le\sum_{k=1}^{n}2^{k}r^{2^{k}-1}\frac{1-a^{2}}{|1-az^{2^{k}}|^{2}}.\label{Basic_inequality}
\end{equation}
Using that $\|(1- b z^N)^{-1}\|_{H^2}^2 = (1 - b^2)^{-1}$ for $b \in [0,1)$, we find that
\[
(1-r)\int_{0}^{2\pi}|B'(re^{it})|dt\le2\pi\sum_{k=1}^{n}2^{k}r^{2^{k}-1}\frac{(1-r)(1-a^{2})}{1-a^{2}r^{2^{k+1}}}\lesssim\frac{1}{n}\sum_{k=1}^{n}2^{k}r^{2^{k}-1}\frac{1-r}{1-ar^{2^{k}}}.
\]
Let us first estimate this quantity for $0\le r\le \frac{1}{2}$. In this case  
\[
\frac{1}{n}\sum_{k=1}^{n}2^{k}r^{2^{k}-1}\frac{1-r}{1-ar^{2^{k}}}\lesssim\frac{1}{n}\sum_{k=1}^{n}
2^{k-2^k} \lesssim\frac{1}{n}.
\]

From now one we assume that $r=1-\frac{1}{2^{s}}$, where $s\ge1$, and we write 
\[
\frac{1}{n}\sum_{k=1}^{n}2^{k}r^{2^{k}-1}\frac{1-r}{1-ar^{2^{k}}}=\frac{1}{n}\sum_{k=1}^{[s]}2^{k}r^{2^{k}-1}\frac{1-r}{1-ar^{2^{k}}}+\frac{1}{n}\sum_{k=[s]+1}^{n}2^{k}r^{2^{k}-1}\frac{1-r}{1-ar^{2^{k}}}=S_{1}+S_{2}.
\]
Since $(1-x)^{t}<e^{-tx}$, $0<x<1$, $t>0$, we have
\[ 
r^{2^{k}}=\bigg(1-\frac{1}{2^{s}}\bigg)^{2^{k}}<e^{-2^{k-s}}.
\]
Therefore, for $k\ge[s]+1$ we have $r^{2^{k}}<e^{-1}$ and so 
\[
S_{2}\lesssim\frac{1}{n}\sum_{k=[s]+1}^{n}2^{k-s}e^{-2^{k-s}}\lesssim\frac{1}{n}.
\]
For $k\le[s]$ we use the inequality 
\[
r^{-2^{k}}-a>e^{2^{k-s}}-1+\frac{1}{n}>2^{k-s}+\frac{1}{n}.
\]
Thus, 
\[
S_{1}\lesssim\frac{1}{n}\sum_{k=1}^{[s]}2^{k}\frac{1-r}{r^{-2^{k}}-a}<\frac{1}{n}\sum_{k=1}^{[s]}\frac{2^{k-s}}{2^{k-s}+\frac{1}{n}}.
\]
We split this sum into two more sums, over $k$ such that $2^{k-s}<\frac{1}{n}$ and
$2^{k-s}\ge\frac{1}{n}$. Then we have 
\[
S_{1}\lesssim\frac{1}{n}\sum_{1\le k<s-\frac{\log n}{\log2}}2^{k-s}\cdot n+\frac{1}{n}\sum_{s-\frac{\log n}{\log2} \le k \le[s]}2^{k-s}\cdot2^{s-k}\lesssim2^{-\frac{\log n}{\log 2}}+\frac{\log n}{n\log2}\lesssim\frac{1}{n}+\frac{\log n}{n\log2}.
\]

Thus, we have shown that 
\[
\|B\|_{B_{1,\infty}^{0}}^{*}\lesssim\frac{1}{n}+\frac{\log n}{n\log2} \lesssim\frac{\log\log N}{\log N}.
\]
\qed

\subsection{Proof of Theorem \ref{thm:low_bd_infty_1}}
 Applying Lemma \ref{lem:lower_bound} to $\sigma=\sigma^{\star}$
with $(p,q)=(\infty,1)$ we obtain 
\[
  \prod_{i=1}^N |\lambda_i| \cdot {\rm cap}_{\mathcal{B}_{\infty,1}^{0}}(\sigma^{\star})\gtrsim\frac{1}{\Norm{B}{\mathcal{B}_{1,\infty}^{0}}^{*}}
\]
because $\prod_{i=1}^{N}\abs{\lambda_{i}}=a^{n} < e^{-1}.$ It remains
to apply Proposition \ref{prop:1}. 
\qed
\bigskip


\section{\label{sec:Proof-of-th_low_bds}Proofs of Theorem \ref{thm:low_bd_p_q} and Theorem \ref{thm:Bpp_lower_bound}}

\subsection{Proof of Proposition \ref{jer}}
 As in the proof of Proposition \ref{prop:1}, for simplicity we
write $B$ instead of $B^{\star}$ throughout the proof.
\medskip
\\
\textit{\uline{Step 1: the case \mbox{$q=\infty$}.}} 
Note that the case $p=1$ is already covered by Proposition \ref{prop:1}.
We start with the case $1<p\le2$. We have to prove that
\begin{equation}
\sup_{0\le r<1}(1-r)\left(\int_{0}^{2\pi}|B\rq{}(re^{it})|^{p}dt\right)^{1/p}\lesssim\frac{1}{(\log N)^{1/p}}\label{inequality1}.
\end{equation}
It follows from (\ref{Basic_inequality}) that 
\[
I:=(1-r)^{p}\int_{0}^{2\pi}|B'(re^{it})|^{p}dt\lesssim\int_{0}^{2\pi}\left(\frac{1}{n}\sum_{k=1}^{n}2^{k}r^{2^{k}-1}\frac{(1-r)}{|1-ar^{2^{k}e^{i2^{k}t}}|^{2}}\right)^{p}dt.
\]
Since for $0<p/2\leq1$ and any $a_{k}\ge0$ one has 
\[
\Big(\sum_{k}a_{k}\Big)^{p}\leq\Big(\sum_{k}a_{k}^{p/2}\Big)^{2},
\]
we conclude that 
\begin{align*}
  I&\leq\frac{1}{n^{p}}\int_{0}^{2\pi}\left(\sum_{k=1}^{n} \frac{(2^{k}r^{2^{k}-1}(1-r))^{p/2}}{|1-ar^{2^{k}}e^{i2^{k}t}|^{p}}\right)^{2}dt \\
   &\lesssim\frac{1}{n^{p}}\int_{0}^{2\pi}\sum_{k\leq j\leq n}\frac{(2^{k}r^{2^{k}-1}(1-r))^{p/2}(2^{j}r^{2^{j}-1}(1-r))^{p/2}}{|1-ar^{2^{k}}e^{i2^{k}t}|^{p}|1-ar^{2^{j}}e^{i2^{j}t}|^{p}}dt \\
   &\lesssim\frac{1}{n^{p}}\int_{0}^{2\pi}\sum_{k\leq j\leq n}\frac{(2^{k}r^{2^{k}-1}(1-r))^{p/2}(2^{j}r^{2^{j}-1}(1-r))^{p/2}}{|1-ar^{2^{k}}e^{i2^{k}t}|^{p}(1-ar^{2^{j}})^{p}}dt.
\end{align*}
After integration with respect to $t$ and using a well-known estimate of Forelli and Rudin (see \cite[Theorem 1.7]{HKZ}) we get
\begin{align*}
  I &\lesssim\frac{1}{n^{p}}\sum_{k\leq j\leq n}\frac{(2^{k}r^{2^{k}-1}(1-r))^{p/2}(2^{j}r^{2^{j}-1}(1-r))^{p/2}}{(1-ar^{2^{k}})^{p-1}(1-ar^{2^{j}})^{p}} \\
    &\lesssim\frac{1}{n^{p}}\sum_{k\leq j\leq n}\frac{(2^{k}r^{2^{k}-1}(1-r))^{p/2}(2^{j}r^{2^{j}-1}(1-r))^{p/2}}{(1-ar^{2^{k}})^{p-1/2}(1-ar^{2^{j}})^{p-1/2}} \\
    &\lesssim\frac{1}{n^{p}}\left(\sum_{k=1}^{n}\frac{(2^{k}r^{2^{k}-1}(1-r))^{p/2}}{(1-ar^{2^{k}})^{p-1/2}}\right)^{2}.
\end{align*}
Thus, we need to show that 
\[
S=\frac{1}{n^{p/2}}\sum_{k=1}^{n}\frac{(2^{k}r^{2^{k}-1}(1-r))^{p/2}}{(1-ar^{2^{k}})^{p-1/2}}\leq\frac{1}{\sqrt{n}}.
\]

If $r\leq 1/2$, then, clearly, $S\lesssim n^{-p/2} \le n^{-1/2}$. 
Now, let $r=1-1/2^{s}$ where $s\ge1$. If $k\ge [s]+1$, then $r^{2^{k}}<e^{-2^{k-s}}\leq1/e$
and 
\[
\frac{1}{n^{p/2}}\sum_{k=[s]+1}^{n}\frac{(2^{k}r^{2^{k}-1}(1-r))^{p/2}}{(1-ar^{2^{k}})^{p-1/2}}\lesssim\frac{1}{n^{p/2}}\sum_{k=[s]+1}^{n}(2^{k-s}e^{-2^{k-s}})^{p/2}\lesssim\frac{1}{n^{p/2}}.
\]
Note that $r^{2^{k}}=(1-1/2^{s})^{2^{k}}\ge(1-1/2)^{2}$ for $k \le [s]$
and, therefore, as in the proof of Proposition \ref{prop:1},
\[
|1-ar^{2^{k}}|\gtrsim r^{-2^{k}}-a= r^{-2^{k}} -1 +1/n\ge 2^{k-s}+1/n.
\]
As in the proof of Proposition \ref{prop:1} we split the sum into
two parts. For $1\leq k<s-\frac{\log n}{\log2}$ we have $2^{k-s}<1/n$
and, therefore,
\begin{align*}
\frac{1}{n^{p/2}}\sum_{k<s-\log n/\log2}\frac{(2^{k}r^{2^{k}-1}(1-r))^{p/2}}{(1-ar^{2^{k}})^{p-1/2}}
&\lesssim \frac{1}{n^{p/2}}n^{p-1/2}\sum_{k<s-\log n/\log2}2^{(k-s)p/2} \\
&\lesssim n^{p/2-1/2}2^{(- \log n/\log2)p/2} = n^{-1/2}.
\end{align*}
Finally, for $s-\frac{\log n}{\log2}\leq k\leq[s]$ we have $2^{k-s}\ge1/n$
and so 
\begin{align*}
\frac{1}{n^{p/2}}\sum_{s-\log n/\log2\leq k\leq[s]}\frac{(2^{k}r^{2^{k}-1}(1-r))^{p/2}}{(1-ar^{2^{k}})^{p-1/2}}
&\lesssim\frac{1}{n^{p/2}}\sum_{s-\log n/\log2\leq k\leq[s]}2^{(k-s)p/2}2^{(s-k)(p-1/2)} \\
&\lesssim n^{-p/2}2^{(\log n/\log2)(p-1)/2}=n^{-1/2}.
\end{align*}
 Thus, we have shown that $S\le n^{-1/2}$ for $1<p\le2$, and so 
\[
  I=(1-r)^{p}\int_{0}^{2\pi}|B'(re^{it})|^{p}dt \lesssim S^{2} \lesssim \frac{1}{\log N}.
\]
The estimate remains true for $p>2$ since by the Schwarz--Pick inequality,
we have $|B'(r e^{ it})|^p \le (1-r^2)^{2-p} |B'(r e^{it})|^2$.
\medskip
\\
\textit{\uline{Step 2: the case \mbox{$1\le p\le q<\infty$}.}}
We have to show that
\[
\int_{0}^{1}(1-r)^{q-1}\left(\int_{0}^{2\pi}|B'(re^{it})|^{p}dt\right)^{q/p}dr\lesssim\frac{1}{(\log N)^{q/p-1}}
\]
(respectively $\lesssim \frac{(\log \log N)^q}{(\log N)^{q-1}}$ in case $p=1$).
It follows from \eqref{inequality1} that for $1<p<\infty$
\[
\int_{0}^{1-1/N^2}(1-r)^{q-1}\left(\int_{0}^{2\pi}|B'(re^{it})|^{p}dt\right)^{q/p}dr\lesssim\frac{1}{(\log N)^{q/p}}\int_{0}^{1-1/N^2}\frac{dr}{1-r}\lesssim\frac{1}{(\log N)^{q/p-1}},
\] 
while for $p=1$ we have by Proposition \ref{prop:1} that
\[
\int_{0}^{1-1/N^2}(1-r)^{q-1}\left(\int_{0}^{2\pi}|B'(re^{it})|dt\right)^{q}dr\lesssim
\frac{(\log\log N)^q}{(\log N)^{q-1}}.
\] 
On the other hand, note that, since $\int_0^{2\pi} |B'(r e^{it})| dt \le \int_0^{2\pi} |B'(e^{it})| dt =2\pi N$, we have
by the Schwarz--Pick inequality
\[
\int_{0}^{2\pi}|B'(re^{it})|^{p}dt\leq\frac{1}{(1-r^{2})^{p-1}}\int_{0}^{2\pi}|B'(re^{it})|dt \le \frac{2\pi N}{(1-r^{2})^{p-1}}.
\]
Therefore, 
\[
\begin{aligned}
\int_{1-1/N^2}^{1}(1-r)^{q-1}  & \left(\int_{0}^{2\pi}|B'(re^{it})|^{p}dt\right)^{q/p}dr 
\lesssim N^{q/p}\int_{1-1/N^2}^{1}(1-r)^{q-1-q(p-1)/p}dr \\
& =N^{q/p}\int_{1-1/N^2}^{1}(1-r)^{q/p-1}\lesssim N^{-q/p} = o\Big(\frac{1}{(\log N)^{q/p-1}} \Big).
\end{aligned}
\]
Combining the above estimates we come to the conclusion of the proposition. 
\qed

\subsection{Proof of Theorem \ref{thm:low_bd_p_q}}
 As in the proof of Theorem \ref{thm:low_bd_infty_1}, we apply Lemma
\ref{lem:lower_bound} to $\sigma=\sigma^{\star}$.  This gives 
\[
  \prod_{i=1}^N |\lambda_i| \cdot 
  {\rm cap}_{\mathcal{B}_{p,q}^{0}}(\sigma^{\star})\gtrsim\frac{1}{\Norm{B}{\mathcal{B}_{p',q'}^{0}}^{*}}
\]
again because $\prod_{i=1}^{N}\abs{\lambda_{i}}=a^{n} < e^{-1}.$ It
remains to apply Proposition \ref{jer} (with $p',q'$ in place of $p,q$) to prove the lower bounds \eqref{eq:low_bd_p_q} 
and \eqref{tte1}. The upper estimate in \eqref{tte2} follows from Theorem
\ref{thm:Upper_bds}. 
\qed


\subsection{Proof of Theorem \ref{thm:Bpp_lower_bound}}

To pass from infinite to finite Blaschke sequences, we require the following continuity property of capacity.

\begin{lem}
  \label{lem:continuity}
  Let $1 \le p,q \le \infty$.
  Let $\sigma \subset \mathbb{D} \setminus \{0\}$ be an infinite sequence. For $N \in \mathbb{N}$, let
  $\sigma_N$ consist of the first $N$ points of $\sigma$. Then
  \begin{equation*}
    \ccap_{\mathcal{B}^0_{p,q}}(\sigma) = \lim_{N \to \infty} \ccap_{\mathcal{B}^0_{p,q}} (\sigma_N).
  \end{equation*}
\end{lem}

\begin{proof}
  For simplicity, we abbreviate $\ccap = \ccap_{\mathcal{B}^0_{p,q}}$.
  The inequality $\ccap(\sigma_N) \le \ccap(\sigma)$ is trivial,
  so it suffices to show that $\ccap(\sigma) \le \liminf_{N \to \infty} \ccap(\sigma_N)$.
  Clearly, we may assume that the limit inferior is finite.

  If $c > \liminf_{N \to \infty} \ccap(\sigma_N)$, then by definition of capacity, there exist a sequence $(N_k)$ tending to infinity
  and functions $f_{N_k} \in \mathcal{B}^0_{p,q}$ such that $f_{N_k}$ vanishes on $\sigma_{N_k}$, $f_{N_k} (0) = 1$
  and $\|f\|_{\mathcal{B}^0_{p,q}} \le c$ for all $k$.
  Then $(f_{N_k})_k$ is a normal family, so a subsequence converges locally uniformly on $\mathbb{D}$ to a holomorphic function $f$.
  By Fatou's lemma, $f \in \mathcal{B}^0_{p,q}$ with $\|f\|_{\mathcal{B}^0_{p,q}} \le c$, and $f$ vanishes on $\sigma$ and $f(0) = 1$. Thus,
  $\ccap(\sigma) \le c$.
\end{proof}

\begin{proof}[Proof of Theorem  \ref{thm:Bpp_lower_bound}]
  Let $2 < q < \infty$. Then, for $f \in \mathcal{B}^0_{\infty,\infty} \cap \mathcal{B}^0_{2,2}$, we have
  \begin{equation*}
    \int_{0}^1 (1 -\rho)^{q-1} \|f'_\rho\|_{L^q}^q \, d \rho \le
    \int_{0}^1 (1  - \rho)^{q-2} \|f'_\rho\|_{L^\infty}^{q-2} (1 - \rho) \|f_\rho'\|_{L^2}^2 \, d \rho
    \le {\|f\|^*}^{q-2}_{\mathcal{B}^0_{\infty,\infty}} {\|f\|^*}^2_{\mathcal{B}^0_{2,2}}.
  \end{equation*}
  Recalling that $\mathcal{B}^0_{2,2} = H^2$ with equivalence of norms, we conclude that there exists a constant $C > 0$
  so that for any Blaschke product $B$, we have
  \begin{equation*}
    \|B\|^*_{\mathcal{B}^0_{q,q}} \le C^{1/q} {\|B\|^*}^{1-\frac{2}{q}}_{\mathcal{B}^0_{\infty,\infty}}.
  \end{equation*}
  (This also follows from the Littlewood--Paley inequality.) Note that the inequality trivially holds for $q=\infty$ as well.
  
  Let $\varepsilon > 0$. It follows from a theorem of Aleksandrov, Anderson, and Nicolau \cite[Theorem 2]{AAN99} 
  that there exists an infinite Blaschke product $B$ with $\|B\|^*_{\mathcal{B}^0_{\infty,\infty}} \le \varepsilon$.
  By precomposing $B$ with a conformal automorphism,
  we may assume that $|B(0)| = \frac{1}{2}$.
  By the preceding estimate with $q = p'$, we have $\|B\|^*_{\mathcal{B}^0_{p',p'}} \le C^{1/p'} \varepsilon^{1 - \frac{2}{p'}}$,
  so Lemma \ref{lem:lower_bound} implies that there exists a constant $c > 0$ such that
  \begin{equation*}
    \frac{1}{2} \ccap_{\mathcal{B}^0_{p,p}}(\sigma) \ge c\varepsilon^{\frac{2}{p'} -1}.
  \end{equation*}
  We distinguish two cases. If $\ccap_{\mathcal{B}^0_{p,p}}(\sigma) < \infty$, then Lemma
  \ref{lem:continuity} yields a finite subsequence $\sigma' \subset \sigma$ such that
  $\ccap_{\mathcal{B}^0_{p,p}}(\sigma') \ge \frac{1}{2} \ccap_{\mathcal{B}^0_{p,p}}(\sigma)$. Note that $\prod_{\lambda \in \sigma'} |\lambda| \ge \prod_{\lambda \in \sigma} | \lambda| = 1/2$, so
  \begin{equation*}
    \prod_{\lambda \in \sigma'} |\lambda| \cdot \ccap_{\mathcal{B}^0_{p,p}}(\sigma') \ge
    c\varepsilon^{\frac{2}{p'} -1}.
  \end{equation*}
  If $\ccap_{\mathcal{B}^0_{p,p}}(\sigma) = \infty$, 
  then Lemma \ref{lem:continuity} directly yields a finite sequence $\sigma' \subset \sigma$ satisfying the last inequality.
  Since $\varepsilon > 0$ was arbitrary, the result follows.
\end{proof}
\bigskip


\section{\label{sec:Proofs-of-th_up_bds}Proofs of the upper bounds }

An obvious choice of a test function for an estimate of the capacity
from above is $f= B_{\sigma}/B_{\sigma}(0)  = (-1)^N \big(\prod_{i=1}^{N}   \lambda_{i} \big)^{-1}B_{\sigma}$,
$\sigma=\left(\lambda_{1},\dots,\lambda_{N}\right)$. Surprisingly,
this function gives sharp estimates in many situations. To use $f$ 
as a test function in the proof of Theorem \ref{thm:Upper_bds} we need first to obtain
some estimates for the norms of finite Blaschke products in Besov spaces
which may be of independent interest.

\subsection{Estimates for the $\mathcal{B}_{p,q}^{0}$ norms of finite Blaschke products} 

\begin{prop}
\label{bln} Let $(p,q)\in[1,\infty]^{2}$, $\sigma=\left(\lambda_{1},\dots,\lambda_{N}\right)\in\mathbb{D}^{N}$
and $B=B_{\sigma}$.

1) If $(p,q)\in[1,2]^{2}$, then 
\[
\|B\|_{\mathcal{B}_{p,q}^{0}}\lesssim(\log N)^{1/q-1/2}.
\]

2) If $1\leq p\leq q\leq\infty$ and $q\geq2$, then 
\[
\|B\|_{\mathcal{B}_{p,q}^{0}}\lesssim1.
\]

3) If $1\leq q\leq p<\infty$ and $p\geq2$, then 
\[
\|B\|_{\mathcal{B}_{p,q}^{0}}\lesssim(\log N)^{1/q-1/p}.
\]

4) If $1\leq q<\infty$ and $p=\infty$, then 
\[
\|B\|_{\mathcal{B}_{p,q}^{0}}\lesssim N^{1/q}.
\]
The constants in the relations $\lesssim$ depend only on $p,q$,
but not on $N$ and $\sigma$. Moreover, in all inequalities the dependence on the growth on $N$ is best possible.
\end{prop}

Note that there is an essential difference between the case $p<\infty$
and $p=\infty$, where the growth is much faster. 

In the proof of Proposition \ref{bln} we will need several simple
estimates, the first of which can be found in \cite{bk}. We will
give their proofs for the sake of completeness. 
\begin{lem}
Let $B$ be a Blaschke product of degree $N$. Then 
\begin{equation}
\int_{0}^{1}(1-\rho)^{p-1}\int_{0}^{2\pi}|B'(\rho e^{it})|^{p}dtd\rho\lesssim\begin{cases}
(\log N)^{1-p/2}, & 1\le p\le2,\\
1, & 2\le p<\infty,
\end{cases}\label{mal1}
\end{equation}
and, for $\rho\in[0,1)$ and $1\le p<\infty$, 
\begin{equation}
\int_{0}^{2\pi}|B'(\rho e^{it})|^{p}dt\lesssim\frac{N}{(1-\rho)^{p-1}}.\label{mal2}
\end{equation}
\end{lem}

\begin{proof}
Since $\int_{0}^{2\pi}|B'(\rho e^{it})|dt\le2\pi N$, $\rho\in[0,1]$,
we conclude that 
\[
\int_{1-1/N}^{1}\int_{0}^{2\pi}|B'(\rho e^{it})|dtd\rho\lesssim1.
\]
Note that 
\begin{equation}
\label{tind}
\int_{0}^{1}\rho(1-\rho^{2})\int_{0}^{2\pi}
|f'(\rho e^{it})|^{2}dtd\rho=\pi\sum_{n=1}^{\infty}\frac{n}{n+1}|a_{n}|^{2} \le
\pi\|f\|_{H^{2}}^{2}\le\pi\|f\|_{H^{\infty}}^{2}
\end{equation}
for any function $f(z)=\sum_{n\ge0}a_{n}z^{n}$ in the Hardy space
$H^{2}$. Therefore, 
\[
\begin{aligned}
&\int_{0}^{1-1/N}\int_{0}^{2\pi}|B'(\rho e^{it})|dtd\rho  \\
  \lesssim &\bigg(\int_{0}^{1-1/N}\int_{0}^{2\pi}(1-\rho)|B'(\rho e^{it})|^{2}dtd\rho\bigg)^{1/2}
           \bigg(\int_{0}^{1-1/N}\int_{0}^{2\pi}\frac{dtd\rho}{1-\rho}\bigg)^{1/2}  \\
  \lesssim &(\log N)^{1/2}.
\end{aligned}
\]
Thus, the inequality is already proved for $p=2$ (simply apply \eqref{tind} to $B$)
and $p=1$. For $1<p<2$ inequality \eqref{mal1} follows from the Hölder inequality
with exponents $(p-1)^{-1}$ and $(2-p)^{-1}$ (note that $p= 2(p-1)+2-p$)
and the estimates for exponents 1 and 2. Finally, for $p>2$ it
follows from the estimate $(1-|z|^{2})|B'(z)|\le1$, $z\in\mathbb{D}$,
that $(1-\rho)^{p-1}\|B'_\rho\|_{L^p(\mathbb{T})}^p \le 
(1-\rho) \|B'_\rho\|_{L^2(\mathbb{T})}^2$ and we can again apply \eqref{tind}.

The estimate \eqref{mal2} is obvious: 
\[
\int_{0}^{2\pi}|B'(\rho e^{it})|^{p}dt\lesssim\frac{1}{(1-\rho)^{p-1}}\int_{0}^{2\pi}|B'(\rho e^{it})|dt\lesssim\frac{N}{(1-\rho)^{p-1}}.
\]
\end{proof}
\begin{proof}[Proof of Proposition \ref{bln}]
Let $1\le p<\infty$. Then it follows from \eqref{mal1} that 
\[
  (\|B\|_{\mathcal{B}_{p,p}^{0}}^*)^{p} =2 \pi \int_{0}^{1}(1-\rho)^{p-1}\int_{0}^{2\pi}|B'(\rho e^{it})|^{p}dtd\rho\lesssim\begin{cases}
(\log N)^{1-p/2}, & 1\le p\le2\\
1, & p\ge2.
\end{cases}
\]
Also trivially $\|B\|_{\mathcal{B}_{\infty,\infty}^{0}}\lesssim1$.

Now let $1\le q\le p<\infty$. We write 
\[
  \begin{aligned}  \frac{1}{2 \pi } (\|B\|_{\mathcal{B}_{p,q}^{0}}^*)^q & =\int_{0}^{1-1/N}(1-\rho)^{q-1}\bigg(\int_{0}^{2\pi}|B'(\rho e^{it})|^{p}dt\bigg)^{q/p}d\rho\\
 & +\int_{1-1/N}^{1}(1-\rho)^{q-1}\bigg(\int_{0}^{2\pi}|B'(\rho e^{it})|^{p}dt\bigg)^{q/p}d\rho=I_{1}+I_{2}.
\end{aligned}
\]
We show that $I_{2}\lesssim1$. Indeed, applying \eqref{mal2} we
get 
\[
I_{2}\lesssim N^{q/p}\int_{1-1/N}^{1}(1-\rho)^{q-1}(1-\rho)^{-q(1-1/p)}d\rho=N^{q/p}\int_{1-1/N}^{1}(1-\rho)^{-1+q/p}d\rho\lesssim1.
\]

To estimate $I_{1}$, we apply the Hölder inequality with exponents
$p/q$ and $p/(p-q)$ to get (with an obvious modification for $p=q$)
\[
I_1 \le\bigg(\int_{0}^{1-1/N}(1-\rho)^{p-1}\int_{0}^{2\pi}|B'(\rho e^{it})|^{p}dtd\rho\bigg)^{\frac{q}{p}}\bigg(\int_{0}^{1-1/N}\frac{d\rho}{1-\rho}\bigg)^{\frac{p-q}{p}}.
\]
Hence, for $1<p\le2$, 
\[
I_1 \lesssim(\log N)^{\frac{q}{p}(1-\frac{p}{2})+\frac{p-q}{p}}=(\log N)^{1-\frac{q}{2}},
\]
while for $p>2$ 
\[
I_1 \lesssim(\log N)^{\frac{p-q}{p}}.
\]
Thus, we have proved 3) and 1) for the case $p\ge q$.

If $1\le p<q<\infty$ we simply have 
\[
\begin{aligned}
\|B\|_{\mathcal{B}_{p,q}^{0}}^* & =\bigg(\int_{0}^{1}(1-\rho)^{q-1}\|B'_{\rho}\|_{L^{p}(\mathbb{T})}^{q}d\rho\bigg)^{1/q}\\
 & \lesssim\bigg(\int_{0}^{1}(1-\rho)^{q-1}\|B'_{\rho}\|_{L^{q}(\mathbb{T})}^{q}d\rho\bigg)^{1/q}\lesssim\begin{cases}
(\log N)^{1/q-1/2}, & 1\le q\le2,\\
1, & q\ge2.
\end{cases}
\end{aligned}
\]
The case $q=\infty$ is trivial by the Schwarz--Pick lemma. The proof of the statements 1)--3)
is completed.

4) Consider the case $p=\infty$. If $B(z)=\prod_{j=1}^{N}\frac{z-\lambda_{j}}{1-\bar{\lambda}_{j}z}$,
then 
\[
B'(z)=\sum_{j=1}^{N}\hat{B}_{j}(z)\frac{1-|\lambda_{j}|^{2}}{(1-\bar{\lambda}_{j}z)^{2}},
\]
where $\hat{B}_{j}(z)=\prod_{k\ne j}\frac{z-\lambda_{k}}{1-\bar{\lambda}_{k}z}$.
Hence,
\[
\|B'_{\rho}\|_{\infty}\lesssim\sum_{j=1}^{N}\frac{1-|\lambda_{j}|^{2}}{(1-|\lambda_{j}|\rho)^{2}}
\]
and, using again the fact that $\|B'_{\rho}\|_{\infty}\le(1-\rho)^{-1}$,
we get 
\[
  (\|B\|_{\mathcal{B}_{\infty,q}^{0}}^*)^{q}=\int_{0}^{1}(1-\rho)^{q-1}\|B'_{\rho}\|_{\infty}^{q}d\rho\le\int_{0}^{1}\|B'_{\rho}\|_{\infty}d\rho\lesssim N.
\]

Let us show that all estimates are sharp. The growth $(\log N)^{1/q-1/p}$ in the case 3) is achieved 
by the product $B^\star$ defined by \eqref{bsta}. Indeed,
for $1\le q\le p <\infty$,
\[
  \|B^\star\|_{\mathcal{B}_{p,q}^{0}}^* \ge \prod_{i=1}^{N}|\lambda_{i}|\cdot{\rm cap}_{\mathcal{B}_{p,q}^{0}}(\sigma^{\star})
\gtrsim (\log N)^{1/q-1/p}
\]
by Theorem \ref{thm:low_bd_p_q}.
In the case 2) the optimality of the estimate can be already seen on $B(z) = z^N$. 

For the case $1 \le p,q \le 2$ one can use an example of a Blaschke product constructed in \cite{bk}:
there exists a Blaschke product of order $N$ such that
$$
\int_0^{1- 1/N} \int_0^{2\pi} |B'(\rho e^{it})|dt d\rho \ge c  \sqrt{\log N},
$$
where $c>0$ is an absolute constant; see the end of Section 2 in \cite{bk}. This construction is based on deep results of R.~Bañuelos
and C.\,N.~Moore \cite{BanMoo} related to Makarov's law of the iterated logarithm.  
An easy application of the H{\"o}lder inequality shows that for $1 \le p,q \le 2$
$$
\int_0^{1- 1/N} (1-\rho)^{q-1} \|B'_\rho\|^q_{L^p(\mathbb{T})} d\rho \gtrsim ( \log N)^{1-q/2}.
$$

Finally, let us show that the estimate in the case 4) also is best possible. Take
$\lambda_j = 1 - 2^{-j}$. Since the sequence $(\lambda_j)$
is an interpolating sequence for $H^\infty$, there exists $\delta > 0$ such that
\begin{equation*}
  \prod_{k \neq j} \Big| \frac{\lambda_k - \lambda_j}{1 - \overline{\lambda_k} \lambda_j} \Big| \ge \delta
\end{equation*}
for all $j$. Thus, if $B$ denotes the Blaschke product with zeros $\lambda_1,\ldots,\lambda_N$, then
\begin{equation*}
  |B'(\lambda_j)| \ge \frac{\delta}{1 - |\lambda_j|^2} \gtrsim 2^j
\end{equation*}
for $j=1,\ldots,N$. It follows that $\|B'_{\rho}\|_\infty \ge 2^j$ for $\rho \ge \lambda_j$ and thus
\[
  \|B\|_{\mathcal{B}_{\infty,q}^{0}}^{q}\gtrsim\sum_{j=1}^{N-1}\int_{\lambda_j}^{\lambda_{j+1}}(1-\rho)^{q-1}\|B'_{\rho}\|_{\infty}^{q}d\rho
  \gtrsim N.
\]
\end{proof}
\medskip

\subsection{Proof of Theorem \ref{thm:Upper_bds}}
Consider the simplest test function 
\[
f=\frac{B_{\sigma}}{B_{\sigma}(0)} =  (-1)^{N}  \frac{B_{\sigma}}{\prod_{i=1}^{N} \lambda_{i}}.
\]
Then 
\[
{\rm cap}_{\mathcal{B}_{p,q}^{0}}(\sigma)\le\frac{\|B_{\sigma}\|_{\mathcal{B}_{p,q}^{0}}}{\prod_{i=1}^{N}\abs{\lambda_{i}}},
\]
and the statement follows from Proposition \ref{bln} in all cases
except $p=\infty$ (in Region III), where an application of Proposition
\ref{bln} will lead to a substantially worse growth order.

To treat the case $p=\infty$ we use the test function from \cite{NN1}. Put $r=1-1/N$ and consider
the finite Blaschke product $\tilde{B}$ with zeros $r\lambda_{1},\dots,r\lambda_{N}$,
\[
\tilde{B}(z)=\prod_{i=1}^{N}\frac{z-r\lambda_{i}}{1-r\overline{\lambda_{i}}z}.
\]
Let
\[
f(z)=(-1)^{N} \frac{\tilde{B}(rz)}{r^{N}\prod_{i=1}^{N}\lambda_{i}}.
\]
Clearly $f$ satisfies $f(0)=1$ and $f(\lambda_{i})=0$ for $i=1,\dots,n.$
Since $r^{N}\asymp1$, we have 
\[
\prod_{i=1}^{N}|\lambda_{i}|\cdot|f'(\rho e^{it})|\lesssim|\tilde{B}'(r\rho e^{it})|\lesssim\frac{1}{1-r\rho}.
\]
Then, for $1\le q<\infty$, 
\[
\bigg(\prod_{i=1}^{N}|\lambda_{i}|\bigg)^{q}\cdot\|f\|_{\mathcal{B}_{\infty,q}^{0}}^{q}\lesssim\int_{0}^{1-1/N}\frac{d\rho}{1-\rho}+\int_{1-1/N}^{1}\frac{d\rho}{1-r}\lesssim\log N.
\]
Theorem \ref{thm:Upper_bds} is proved. 
\qed

\end{document}